\documentclass[12pt,reqno]{amsart}
\usepackage{amsmath,amssymb,amsfonts,amsthm}
\usepackage[left=3cm, right=3cm, top=2.8cm, bottom=2.8cm]{geometry}
\usepackage[utf8]{inputenc}
\usepackage[T1]{fontenc}
\usepackage{color}
\usepackage{multirow}
\usepackage{graphicx}
\usepackage{hyperref}
\usepackage{graphicx}
\usepackage{calc}%
\usepackage{hyperref}
\hypersetup{
    colorlinks=true,
    linkcolor=blue,
    filecolor=magenta,
    urlcolor=cyan,
}

\usepackage{stackrel} 

\newcommand{\capto}{{}^C \! D_t}
\newcommand{\capt}{{}^C \! D^{\alpha}_t}

\DeclareMathOperator*{\esssup}{ess\,sup}
\newtheorem{theorem}{Theorem}
\newtheorem{lemma}[theorem]{Lemma}
\newtheorem{proposition}[theorem]{Proposition}

\newtheorem{definition}[theorem]{Definition}
\newtheorem{remark}[theorem]{Remark}

\theoremstyle{plain}

\begin{document}

\title[A General Version of Carath\'{e}odory's Theorem]{A General Version of Carath\'{e}odory's Existence and Uniqueness Theorem}


\author[P.M. Carvalho-Neto]{Paulo M. de Carvalho-Neto}
\address[Paulo M. de Carvalho Neto]{Department of Mathematics, Federal University of Santa Catarina, Florian\'{o}polis - SC, Brazil}
\email[]{paulo.carvalho@ufsc.br}

\author[C.L. Frota]{C\'{i}cero L. Frota}
\address[C\'{i}cero L. Frota]{Department of Mathematics, State University of Maringá, Maring\'{a} - PR, Brazil\vspace*{0.3cm}}
\email[]{clfrota@uem.br}

\author[P.G.P. Torelli]{Pedro G. P. Torelli}
\address[Pedro G. P. Torelli]{Department of Mathematics, State University of Maringá, Maring\'{a} - PR, Brazil\vspace*{0.3cm}}
\email[]{pgptorelli@gmail.com}






\begin{abstract}
In this paper, we establish a general version of Carath\'{e}odory's existence and uniqueness theorem for a semilinear system of integro-differential equations arising from differential equations with distinct orders of Caputo fractional derivative. The main result of our work demonstrates that the integrability order of the Carath\'{e}odory function $f$ must be at least greater than the maximum of the reciprocals of all differentiation orders in the system; otherwise, even the existence of a solution cannot be guaranteed.
\end{abstract}

\maketitle

\section{Introduction and Some Preliminars}

Given $T>0$ and $\xi\in\mathbb{R}^n$, $n \geq 1$, we consider the classical Cauchy problem
\begin{equation}\label{origedo}
\left\{\begin{array}{ll}
u'(t) = f(u(t),t), & t \in (0, T], \\
u(0) = \xi,
\end{array}\right.
\end{equation}
where $f: \mathbb{R}^n \times[0,T] \subset \mathbb{R}^{n+1} \to \mathbb{R}^n$ is a continuous function. The existence and uniqueness of solution to \eqref{origedo} is a fundamental topic in the theory of differential equations, playing a crucial role in both theoretical and applied contexts. Despite the extensive literature on this subject (see \cite{AgLa1, Li1, Os1, Pi1, To1} for classical references), it remains far from being fully understood.

In the scalar and autonomous case, the existence and uniqueness of solutions to \eqref{origedo} can, in some sense, be fully characterized by necessary and sufficient conditions, as demonstrated by Biding in \cite{Bi1}. 

For the scalar and nonautonomous case, a general condition was established by Levy \cite{Le1}, who proved that if there exists a differentiable function $u : (0,T] \to (0,\infty)$ with $u'(t) > 0$ for all $t \in (0,T]$ such that the following Lipschitz-type inequality holds:
\begin{equation*}
|f(x,t) - f(y,t)| \leq \left(\frac{u'(t)}{u(t)}\right) |x - y|, \quad \forall t \in (0,T], \quad \forall \, x, y \in \Omega,
\end{equation*}
then the necessary and sufficient condition for uniqueness of solutions to \eqref{origedo} is given by $\lim_{t\to 0}(u(t)/t) \not= 0$.

While the results of Biding and Levy apply to a broad class of functions, it is typically assumed that $f$ is continuous, which ensures that the derivative of the unique solution to \eqref{origedo} is also continuous. However, when $f$ is merely measurable, the analysis of existence and uniqueness becomes significantly more challenging. 

Despite these challenges, solutions can still be rigorously formulated under weaker assumptions on $f$. In 1918, in his celebrated work \cite{BiBi1}, Carathéodory introduced a broader class of functions, enabling the definition of solutions even when $f$ lacks continuity. His study laid the foundation for further research that extended classical existence and uniqueness results. It remains widely used and can be stated as follows (cf. \cite[Section 1.2]{AmPr1} and \cite[Section 1.3]{Ro1}):
\begin{definition} We say that $f:\mathbb{R}^n \times[0,T] \subset \mathbb{R}^{n+1} \to \mathbb{R}^n$ is a Carathéodory function, if it satisfies:
\begin{itemize}
\item[(i)] $x\mapsto f(x,t)$ is continuous for almost every $t\in[0,T]$;\vspace*{0.2cm}
\item[(ii)] $t\mapsto f(x,t)$ is Lebesgue measurable for every $x\in\Omega$.
\end{itemize}
\end{definition}

It is classical, however, that the condition of $f$ being a Carath\'{e}odory function is not enough to ensure uniqueness of solutions to \eqref{origedo}, as can be seen by considering $f:[0,1]\times[0,1]\rightarrow\mathbb{R}$ given by
$$f(x,t)=\left\{\begin{array}{ll}2\sqrt{x},&\textrm{ for }t\in[0,1/2],\vspace*{0.1cm}\\\sqrt{x},&\textrm{ for }t\in(1/2,1].\end{array}\right.$$
Then the absolutely continuous functions $\phi_1, \phi_2 : [0,1] \rightarrow \mathbb{R}$ given by
$$\phi_1(t)=\left\{\begin{array}{ll}t^2,&\textrm{ for }t\in[0,1/2],\vspace*{0.1cm}\\t^2/4,&\textrm{ for }t\in(1/2,1],\end{array}\right.$$
and $\phi_2(t) = 0$, are two distinct absolutely continuous solutions of \eqref{origedo}, when $\xi=0$.

To address this issue, we introduce a new condition, which we state below along with the result that ensures the existence and uniqueness of a solution within this more general framework. For a proof of this result, see \cite[Theorem 5.3]{Ha1} (cf. \cite[Theorem 1.45]{Ro1}).

\begin{theorem}\label{classiccaratheodory} Assume that $f:\mathbb{R}^n\times[0,T] \subset \mathbb{R}^{n+1} \rightarrow \mathbb{R}^n$ is a Carathéodory function, satisfying:
\begin{itemize}
\item[($C_1$)] There exist $\gamma \in L^1(0,T)$ and $C>0$ such that
\begin{equation*}
\left\Vert f(x,t) \right\Vert_{\mathbb{R}^n} \leq C\left\Vert x \right\Vert_{\mathbb{R}^n} + \gamma(t),
\end{equation*}
for all $x\in\mathbb{R}^n$ and almost every $t\in[0,T]$;\vspace*{0.2cm}
\item[($C_2$)] There exists a function $l \in L^1(0,T)$ such that
$$\left\Vert f(x,t)-f(y,t)\right\Vert_{\mathbb{R}^n}\leq l(t) \left\Vert x-y \right\Vert_{\mathbb{R}^n},$$
for all $x\in\mathbb{R}^n$ and almost every $t\in[0,T]$.
\end{itemize}
Then, for each $\xi \in \mathbb{R}^n$, there exists a unique absolutely continuous solution $\phi:[0,T]\rightarrow\mathbb{R}^n$ of \eqref{origedo}, i.e., $\phi'(t) = f(\phi(t),t)$ for almost every $t\in[0,T]$ and $\phi(0)=\xi$. 

\end{theorem}

Now we turn our attention to the fractional Cauchy problem given by
\begin{equation}\label{fracclassical}\left\{\begin{array}{ll}
\capt u(t) = f(u(t),t), & \text{ for a.e. } t \in [0, T], \\
u(0) = \xi,
\end{array}\right.\end{equation}
where $\alpha\in(0,1)$ and $\capt$ denotes the Caputo fractional derivative of order $\alpha$.

As in the case $\alpha = 1$, the fractional Cauchy problem \eqref{fracclassical} has been extensively studied in the literature when $f$ is a continuous function; see, for instance, \cite{KiSrTr1,SaKiMa1} for classical references on this subject. However, the case where $f$ is merely a Carathéodory function remains less explored. 

Recently, Lan and Webb \cite{We1} introduced a novel Bihari-type inequality for singular kernels, which arises in the integral formulation of \eqref{fracclassical}. Using this inequality, they proved the existence and uniqueness of a positive solution to \eqref{fracclassical}, assuming that $f$ is a Carathéodory function that, when composed with a continuous function belongs to $L^p$ and satisfies certain growth conditions. A key aspect of their results is the requirement that $\alpha > 1/p$, which is crucial for applying their techniques. This condition is closely linked to the behavior of the Riemann-Liouville fractional integral of order $1/p$ in $L^p$, which does not necessarily gives in a continuous function. This issue was originally analyzed by Hardy and Littlewood \cite{HaLi} and later revisited by Carvalho-Neto and Fehlberg Júnior \cite{CarFe2}. 

In contrast to Lan and Webb's approach, our work focuses on a more general system of differential equations. Specifically, we consider $\{\alpha_j\}_{j=1}^n \subset (0,1]$ and discuss the following system of fractional differential equations in $\mathbb{R}^n$:
\begin{equation}\label{caratheq}
\left\{
\begin{array}{cccl}
\capto^{\alpha_1} \varphi_1(t) &=& f_1\big(\varphi_1(t), \cdots , \varphi_n(t),t\big), & \text{for a.e. } t \in [0,T], \\
\vdots& & \vdots & \\
\capto^{\alpha_n} \varphi_n(t) &=& f_n\big(\varphi_1(t), \cdots , \varphi_n(t),t\big), & \text{for a.e. } t \in [0,T],
\end{array}
\right.
\end{equation}
where $\varphi = (\varphi_1 , \cdots , \varphi_n): [0,T] \to \mathbb{R}^n$ is the unknown function. These equations are subject to the initial conditions:
\begin{equation}\label{carathin}
\left\{
\begin{array}{c}
\varphi_1(0) = \xi_1, \\
\vdots \\
\varphi_n(0) = \xi_n,
\end{array}
\right.
\end{equation}
where $\xi = (\xi_1, \cdots, \xi_n) \in \mathbb{R}^n$ is given. 

To the best of our knowledge, discussions on the existence and uniqueness of solutions to \eqref{caratheq}-\eqref{carathin} are scarce. Recently, we came across the term "multi-order fractional differential equations," which some researchers use to describe similar problems related to \eqref{caratheq}-\eqref{carathin}. Therefore, we find it important to present the most relevant literature that has come to our attention on this topic.

Diethelm and Ford in \cite{DieFo1} addressed a problem, which can be seen as a special case of \eqref{caratheq}-\eqref{carathin} when expressed in matrix form. They required the derivative orders to be rational numbers with some specific conditions, and proved the existence and uniqueness of solutions under the assumption that $f$ is a Lipschitz function.

Faghih and Mokhtary \cite{FaMo1} also considered \eqref{caratheq}-\eqref{carathin} with $f$ being a Lipschitz function and the derivative orders as rational numbers. They proved existence and uniqueness of solution and analyzed how perturbed data affects the solution's behavior and smoothness under certain assumptions.

In our previous work \cite{CarToLo1}, we investigated the time-fractional wave equation with acoustic boundary conditions, which eventually led us to study a system of fractional differential equations with distinct orders. Unlike previous studies, we did not restrict the orders to be rational numbers. In that work, we proved the existence and uniqueness of a solution to the problem \eqref{caratheq}-\eqref{carathin}, when $f$ is a continuous and Lipschitz function on the first variable.

In this paper, our main objective is to prove that if $f$ is an $L^p$-Carath\'{e}odory function (see Definition \ref{lpcarathdef}) and the smallest order of differentiation is greater than $1/p$, then the problem \eqref{caratheq}–\eqref{carathin} admits a unique continuous solution. Moreover, we show (cf. Theorem \ref{202505021500}) that condition $(C_2)$ is not only sufficient but also necessary; its violation may lead to the nonexistence of solutions. A similar phenomenon was already observed in the study by Lan and Webb, when all the orders of differentiation are equal.

To conclude this introductory section, we outline the organization of the paper. In Section 2, we fix the notation and introduce key concepts, such as the Riemann–Liouville fractional integral and the Caputo fractional derivative, along with several technical results that are used throughout the manuscript. 

Section 3 presents our main contributions: an integral formulation of the problem \eqref{caratheq}-\eqref{carathin} (Proposition \ref{integralequ}); a key new inequality involving the Riemann–Liouville fractional integral (Lemma \ref{lemaint}); the existence and uniqueness of a solution to the problem \eqref{caratheq}-\eqref{carathin} (Theorem \ref{202504011725}); and the nonexistence of a solution when one of the orders is exactly $1/p$ (Theorem \ref{202505021500}).

\section{Notations and Results}

This section aims to recall some classical definitions and results concerning vector-valued functions, as well as to present some less classical yet straightforward results that will be used later in this manuscript. For further reading on these topics, we refer to Arendt et al.~\cite{ArBaHiNe1}, Carvalho-Neto et al.~\cite{CarFe0}, and Diestel et al.~\cite{DiUh1}.  

Let us begin setting the function spaces recurrently used in this work. For $n\in\mathbb{N}$, we denote the space $C([0,T];\mathbb{R}^n)$ as the set of all continuous functions $f:[0,T]\to \mathbb{R}^n$. When equipped with the norm  
$$\|f\|_{C([0,T];{\mathbb{R}^n})}:=\sup\big\{\|f(t)\|_{\mathbb{R}^n}:t\in[0,T]\big\},$$  
it forms a Banach space.

For $1\leq p\leq\infty$, we denote by $L^{p}(0,T;\mathbb{R}^n)$ the space of all measurable functions $f:[0,T]\to \mathbb{R}^n$ such that $\|f(\cdot)\|_{\mathbb{R}^n}$ belongs to $L^{p}(0,T)$. Furthermore, when equipped with the norm
$$
\|f\|_{L^p(0,T;{\mathbb{R}^n})}:=\left\{\begin{array}{ll}
\bigg[\displaystyle\int_{0}^{T}{\|f(s)\|^p_{\mathbb{R}^n}}\,ds\bigg]^{1/p},&\textrm{ if }p\in[1,\infty),\vspace*{0.3cm}\\
\esssup_{s\in [0,T]}\|f(s)\|_{\mathbb{R}^n},&\textrm{ if }p=\infty,
\end{array}\right.
$$
it forms a Banach space.

For $1\leq p\leq\infty$, we denote by $W^{1,p}(0,T;{\mathbb{R}^n})$ the subspace of $L^{p}(0,T;{\mathbb{R}^n})$ of every function $f:[0,T]\rightarrow \mathbb{R}^n$ that has a weak derivatives in $L^{p}(0,T;{\mathbb{R}^n})$. By considering the norm
$$\|f\|_{W^{1,p}(0,T;{\mathbb{R}^n})}:=\big\|f\big\|_{L^{p}(0,T;{\mathbb{R}^n})}+\big\|f^\prime\big\|_{L^{p}(0,T;{\mathbb{R}^n})},$$
the space $W^{1,p}(0,T;{\mathbb{R}^n})$ becomes a Banach space. 

Now, we recall the notions of the Riemann-Liouville fractional integral and the Caputo fractional derivative, which are the main tools used in this manuscript. For a more detailed survey on this topics, we refer to \cite{CarFe0,KiSrTr1,SaKiMa1}.

\begin{definition}
	Suppose $\alpha >0$ and $f\in L^{1}(0,T;\mathbb{R}^n)$. The Riemann-Liouville fractional integral of order $\alpha$ of the function $f$, denoted by $J_{t}^{\alpha}f(t)$, is given by
	\begin{equation*}
		J_{t}^{\alpha}f(t):=\dfrac{1}{\Gamma(\alpha)}\int_0^t{(t-s)^{\alpha-1} f(s)}\,ds,\quad \textrm{ for a.e. }t\in[0,T].
	\end{equation*}
\end{definition}

\begin{definition}
	Let $\alpha\in(0,1)$ and $f\in C([0,T],\mathbb{R}^n)$ such that $J^{1-\alpha}_t f \in W^{1,1}(0,T;\mathbb{R}^n)$. The Caputo fractional derivative, of order $\alpha$, of the function $f$, denoted as $\capt f(t)$, is given by
	\begin{equation*}
		\capt f(t):=\dfrac{d}{dt}\Big\{J_{t}^{1-\alpha}\big[f(t)-f(0)\big]\Big\},\;\;\;\;\;  \textrm{ for a.e. }t\in[0,T],
	\end{equation*}
where above $(d/dt)$ represents the weak derivative.
\end{definition}

To conclude this section, we present two results, along with appropriate references, that will be essential for the development of our main results in the following section.

\begin{proposition}[{\cite[Proposition 2.35]{Car1}}]\label{202504011515}
	Let  $\alpha\in(0,1)$ and $f\in C([0,T],\mathbb{R}^n)$. If $J_t^{1-\alpha}f\in W^{1,1}(0,T;\mathbb{R}^n)$, then $J_{t}^{\alpha}\big[\capt h(t)\big]=h(t)-h(0),\textrm{ for a.e. }t\in[0,T].$
\end{proposition}

\begin{theorem}[{\cite[Theorem 7]{CarFe3}}]\label{202504011518} Consider $p\in(1,\infty)$, $\alpha\in(1/p,\infty)$ and assume that $f\in L^p(0,T;\mathbb{R}^n)$. Then $J_{t}^\alpha f\in C([0,T];\mathbb{R}^n )$.
\end{theorem}

\section{The General Caratheodory's Theorem}

In this section, we establish the existence and uniqueness of a solution to \eqref{caratheq}-\eqref{carathin} under the assumption that $f: \mathbb{R}^n \times [0,T] \subset \mathbb{R}^{n+1} \to \mathbb{R}^n$ is a Carathéodory function satisfying condition $(C_2)$ from Theorem \ref{classiccaratheodory}, replacing $(C_1)$ with the following assumption:
\begin{itemize}
\item[($C_p^*$)] For some $1\leq p<\infty$, there exist $\gamma \in L^p(0,T)$ and a constant $C>0$ such that
\begin{equation*}
\left\| f(x,t) \right\|_{\mathbb{R}^n} \leq C\|x\|_{\mathbb{R}^n} + \gamma(t),
\end{equation*}
for all $x\in\mathbb{R}^n$ and almost every $t\in[0,T]$.\vspace*{0.2cm}
\end{itemize}

\begin{definition} \label{lpcarathdef}
A Carathéodory function $f: \mathbb{R}^n \times [0,T] \subset \mathbb{R}^{n+1} \to \mathbb{R}^n$ that satisfies both conditions $(C_p^*)$ and $(C_2)$ is called an $L^p$-Carathéodory function.\end{definition}

\begin{remark}\label{202505061925} Notice that an $L^1$-Carath\'{e}odory function is simply a Carathéodory function that satisfies both conditions $(C_1)$ and $(C_2)$. The $L^p$-Carath\'{e}odory notion is necessary to address the regularity issues introduced by the Riemann--Liouville fractional integral, as extensively discussed in the works of Carvalho-Neto and Fehlberg J\'{u}nior \cite{CarFe0,CarFe1,CarFe2,CarFe3}.
\end{remark}

Before proceeding with the study of existence and uniqueness, we recall a classical result concerning the Nemytskii operator (cf. \cite[Theorem 2.2]{AmPr1} and \cite[Theorem 1.27]{Ro1}).

\begin{proposition} \label{teonemytskii2}
Let $f: \mathbb{R}^n \times [0,T] \to \mathbb{R}^n$ be an $L^p$-Carathéodory function and suppose $u\in L^p(0,T;\mathbb{R}^n)$. Then the Nemytskii function $\mathcal{N}_f(u): [0,T] \to \mathbb{R}^n$ defined by
\begin{equation*}
\mathcal{N}_f(u)(t):= f(u(t),t),
\end{equation*}
belongs to $L^p(0,T;\mathbb{R}^n)$.
\end{proposition}

Now, we formalize the notion of a solution to the Cauchy problem \eqref{caratheq}-\eqref{carathin}, following the classical framework for ordinary differential equations in $\mathbb{R}^n$.

\begin{definition}
A function $\varphi:[0,T]\rightarrow \mathbb{R}^n$ is said to be a solution of the Cauchy problem \eqref{caratheq}-\eqref{carathin} on $[0,T]$, when $f$ is a Carathéodory function, if it satisfies the following conditions:
\begin{itemize}
\item[(i)] $\varphi \in C([0,T];\mathbb{R}^n)$ and $(\capto^{\alpha_1}\varphi_1,\ldots,\capto^{\alpha_n}\varphi_n) \in L^p(0,T;\mathbb{R}^n)$,
\item[(ii)] $\varphi$ satisfies the equations given in \eqref{caratheq}-\eqref{carathin}.
\end{itemize}

\end{definition}

With the aid of Proposition \ref{teonemytskii2}, we can  prove an integral formulation to \eqref{caratheq}-\eqref{carathin}.

\begin{proposition}\label{integralequ}
Suppose $n \in \mathbb{N}$, $\{\alpha_j\}_{j=1}^n \subset (0,1]$ and $p > \max\{1/\alpha_j : j \in \{1, \dots, n\}\}$. Let $f: \mathbb{R}^n \times [0,T] \to \mathbb{R}^n$ be an $L^p$-Carathéodory function. Then a function $\varphi$ is a solution to the system \eqref{caratheq}-\eqref{carathin} on $[0,T]$ if, and only if, it satisfies the system of integral equations
\begin{equation} \label{eqint}
\varphi_j(t) = \xi_j + \dfrac{1}{\Gamma(\alpha_j)} \int_{0}^{t} (t-s)^{\alpha_j - 1} f_j\big(\varphi(s),s\big) \, ds, \quad \forall t \in [0,T],
\end{equation}
for every $j \in \{1,\dots,n\}$.
\end{proposition}
\begin{proof}
Let $\varphi$ be a solution of \eqref{caratheq}-\eqref{carathin} on $[0,T]$. If $\alpha_j = 1, \ j \in \{1, \dots , n\}$, we have the classical case and the result follows directly. We therefore consider the fractional case and assume that $\alpha_j \in (0,1)$. Under this assumption, $\capto ^{\alpha_j} \varphi_j\in L^p(0,T;\mathbb{R})$, for every $1 \leq j \leq n$.  This implies that $J_t^{1-\alpha_j} [\varphi_j(\cdot)-\varphi_j(0)]$ belongs to $W^{1,p}(0,T;\mathbb{R})$, for every $1\leq j\leq n$. Consequently, by applying $J^{\alpha_j}_t$ to the $j$-th equation in \eqref{caratheq} and using Proposition \ref{202504011515}, we establish that
\begin{equation*}
\varphi_j(t) - \varphi_j(0) = J^{\alpha_j}_t f_j(\varphi(t),t), \quad \forall t \in [0,T],
\end{equation*}
for every $1\leq j\leq n$. Considering the initial conditions in \eqref{carathin}, we can deduce \eqref{eqint}.

On the other hand, we assume that \eqref{eqint} holds. Since Proposition \ref{teonemytskii2}  allows us to deduce that $f_j(\varphi(\cdot),\cdot)\in L^p(0,T;\mathbb{R})$, for every $1\leq j\leq n$, we apply Theorem \ref{202504011518} to ensure that $J^{\alpha_j}_tf_j(\varphi(\cdot),\cdot) \in C([0,T];\mathbb{R})$, therefore, $ \varphi_j \in C([0,T];\mathbb{R})$, by \eqref{eqint}.

Now, we employ Hölder's inequality to see that 
\begin{multline*}
\left\vert J^{\alpha_j}_t f_j(\varphi(t),t) \right\vert \leq \frac{1}{\Gamma(\alpha_j)} \int_0^t (t-s)^{\alpha_j-1} \left\vert f(\varphi(s),s) \right\vert ds\\
\leq \left( \int_0^t (t-s)^{(\alpha_j-1) \frac{p}{p-1}} ds \right)^{\frac{p-1}{p}}  \left\Vert f(\varphi(\cdot),\cdot) \right\Vert_{L^p(0,T;\mathbb{R}^n)}
 =  \frac{t^{ \alpha_j -\frac{1}{p}} }{c}   \left\Vert f(\varphi(\cdot),\cdot) \right\Vert_{L^p(0,T;\mathbb{R}^n)},
\end{multline*}
where $c:= \left[({\alpha_j p -1})/({p-1}) \right]^{\frac{p-1}{p}}$.

Then, since $\alpha_j > 1/p$, for every $1\leq j\leq n$, we can take the limit as $t \to 0$ in the inequality above and deduce that $J^{\alpha_j}_t f_j(\varphi(t),t)\big|_{t=0} = 0$. As a consequence, from \eqref{eqint}, it follows that $\varphi_j(0) = \xi_j$, for every $j \in \{1, \dots, n\}$. Therefore, the initial condition \eqref{carathin} is satisfied.

Also observe that \eqref{eqint} implies
\begin{equation*}
J^{1-\alpha_j}_t \left[\varphi_j(t)-\varphi_j(0) \right] = J^{1-\alpha_j}_t \Big[J^{\alpha_j}_t f_j(\varphi(t),t)\Big] = J^1_t  f_j(\varphi(t),t),
\end{equation*}
for every $j \in \{1, \dots, n\}$.

Since $J^1_t f_j(\varphi(\cdot),\cdot) \in W^{1,p}(0,T;\mathbb{R})$, for every $j \in \{1, \dots, n\}$, it follows from the identity above that \eqref{caratheq} is satisfied.
\end{proof}

\begin{remark} \label{202504290836}
In the result above, the case where $\alpha_j = 1$ for every $j \in \{1, \dots, n\}$ corresponds to the classical setting. In this case, we may choose $p = \max\{1/\alpha_j : j \in \{1, \dots, n\}\} = 1$, an option that is not available in the fractional formulation due to the regularity properties of the RL fractional integral. For further details on the regularity of the RL fractional integral, we refer to \cite{CarFe3}.
\end{remark}

Before presenting our main result, we establish the following auxiliary lemma, which is noteworthy in its own right and requires a substantial amount of work to be proved.

\begin{lemma} \label{lemaint}
Let $\rho\in(0,1)$, $q\in(1/\rho,\infty]$ and suppose that $g \in L^q(0,T;\mathbb{R})$ with $g(t) \geq 0$ for almost every $t \in [0,T]$. Then, there exists $n_0 \in \mathbb{N}$ such that
\begin{equation*}
\underbrace{J^\rho_t \Big\{ g(t) J^\rho_t \Big[ g(t) J^\rho_t\Big(\cdots g(t)J^\rho_t g(t)\Big) \Big] \Big\}  }_{n_0 \text{-times}} < 1, \quad \textrm{ for a.e. } t \in [0,T].
\end{equation*}
\end{lemma}

\begin{proof}
At first, let us we make some preliminary observations and introduce convenient notations. 
\begin{itemize}
\item[(i)] Since $q > 1/\gamma$, it follows that $J^\rho_t g \in C([0,T];\mathbb{R})$ (see Theorem \ref{202504011518}).\vspace*{0.2cm}
\item[(ii)] For simplicity, let us denote $\beta := (\rho q - 1)/(q - 1)$. Since $\rho\in(1/q,1)$, it follows that $\beta\in(0,1)$.\vspace*{0.2cm}
\item[(iii)] Recall that the Hölder conjugate of $q$ is given by $q^* := q/(q - 1)$. From this, we deduce the identities
$$
(\rho - 1)q^* = \beta - 1 \quad \text{and} \quad \big[\rho - (1/q)\big] q^* = \beta.
$$
\end{itemize}

We are now ready to begin the inductive argument, which forms the core of the proof. For $n = 1$, applying H\"{o}lder's inequality, we obtain:
\begin{multline*}
J^\rho_t g(t) = \frac{1}{\Gamma(\rho)}\int_0^t (t-s)^{\rho-1} g(s)ds
 \leq  \frac{\left\Vert g \right\Vert_{L^q(0,T;\mathbb{R})}}{\Gamma(\rho)} \left( \int_0^t (t-s)^{(\rho-1)q^*} ds \right)^{{1}/{q^*}}  \\
 = \left(\frac{\left\Vert g \right\Vert_{L^q(0,T)}}{\Gamma(\rho)}\right)\left( \frac{\Gamma(\beta)}{ \beta \Gamma(\beta)} \right)^{{1}/{q^*}}   t^{\rho-\frac{1}{q}},
\end{multline*}
for every $t\in[0,T]$.

For $n = 2$, using the conclusion obtained in the case $n = 1$, we obtain
\begin{multline*}
J^\rho_t \big[ g(t) J^\rho_t g(t) \big] = \frac{1}{\Gamma(\rho)} \int_0^t (t-s)^{\rho-1} g(s)J^\rho_s g(s) ds 
\\\leq  \left(\frac{\left\Vert g \right\Vert^2_{L^q(0,T)}}{\Gamma(\rho)^2}\right)\left( \frac{\Gamma(\beta)}{ \beta \Gamma(\beta)} \right)^{{1}/{q^*}} \left(\int_0^t (t-s)^{(\rho-1)q^*}s^{[\rho-(1/q)]q^*}ds\right)^{1/q^*},
\end{multline*}
for every $t\in[0,T]$, and consequently,
$$
J^\rho_t \big[ g(t) J^\rho_t g(t) \big] \leq \left( \frac{\left\Vert g \right\Vert_{L^q(0,T)}}{\Gamma(\rho)} \right)^2 \left( \frac{\Gamma(\beta)^2}{ 2\beta \Gamma(2\beta)} \right)^{{1}/{q^*}}  t^{2 [\rho-(1/q)]}. 
$$
for every $t\in[0,T]$.

In the general case, we estimate
\begin{equation*}
\underbrace{J^\rho_t \{ g(t) J^\rho_t [ g(t) \cdots J^\rho_t g(t) ] \}  }_{n\textrm{-times}} \leq \left[ \frac{\left\Vert g \right\Vert_{L^q(0,T)}}{\Gamma(\rho)} \right]^n \left( \frac{\Gamma(\beta)^n}{ n \beta \Gamma(n \beta)} \right)^{{1}/{q^*}}  t^{n [\rho-(1/q)]},
\end{equation*}
for every $t\in[0,T]$. Then if we define
$$C_n:=  \left[ \frac{\left\Vert g \right\Vert_{L^q(0,T)}}{\Gamma(\gamma)} \right]^n \left( \frac{\Gamma(\beta)^n}{ n\beta \Gamma(n\beta)} \right)^{{1}/{q^*}}T^{n [\gamma-(1/q)]},$$
we conclude that
$$\underbrace{J^\rho_t \{ g(t) J^\rho_t [ g(t) \cdots J^\rho_t g(t) ] \}  }_{n\textrm{-times}} \leq C_n,$$
for every $t\in[0,T]$.

Let us now prove that the series $\sum_{n=1}^\infty C_n$ is convergent. At first, observe that  
$$
\frac{C_{n+1}}{C_n} = \left[ \frac{\left\Vert f \right\Vert_{L^q(0,T)}T^{\rho-(1/q)}}{\Gamma(\rho)} \right] \left( \frac{\Gamma(\beta) \, n \, \Gamma(n \beta)}{ (n+1) \, \Gamma((n+1) \beta)} \right)^{1/q^*}.
$$
Therefore, to prove the convergence of the series, we need to analyze the behavior of the above term as $n \to \infty$. By \cite{WENDEL} we have 
\begin{equation*}
\left( \frac{x}{x+a} \right)^{1-a} \leq \frac{\Gamma(x+a)}{x^a \Gamma (x)} \leq 1, \quad \text{for all } x>0 \text{ and } 0 < a \leq 1,
\end{equation*}
or equivalently,
\begin{equation*}
\frac{1}{x^a}  \leq \frac{\Gamma(x)}{ \Gamma (x+a)} \leq \frac{(x+a)^{1-a}}{x}.
\end{equation*}

Hence, 
\begin{equation*}
\frac{\Gamma(n\beta)}{\Gamma((n+1)\beta)} \leq \frac{[(n+1)\beta]^{1-\beta}}{n \beta},
\end{equation*}
what allow us to deduce that
\begin{equation*}
B_n := \frac{n\Gamma(n\beta)}{(n+1)\Gamma((n+1)\beta)} \leq \frac{n[(n+1)\beta]^{1-\beta}}{(n+1)n \beta} = \frac{1}{[(n+1)\beta]^\beta}.
\end{equation*}

Therefore, $B_n \leq {1}/{[(n+1)\beta]^\beta} \to 0$ as $n \to \infty$. This implies the desired result. The convergence of the series guarantees the existence of $n_0 \in \mathbb{N}$ such that $C_{n_0}< 1$.
\end{proof}

We now have at our disposal all the necessary tools, concepts, and definitions introduced throughout this work to prove our main result, which establishes a new contribution to the theory of fractional differential equations, particularly in the setting of Carathéodory-type functions, where such general results are still scarcely explored in the literature.

\begin{theorem}\label{202504011725} Suppose $n \in \mathbb{N}$, $\{\alpha_j\}_{j=1}^n \subset (0,1]$ and $p > \max\{1/\alpha_j : j \in \{1, \dots, n\}\}$. Let $f: \mathbb{R}^n \times [0,T] \to \mathbb{R}^n$ be an $L^p$-Carathéodory function. Then, for each $\xi \in \mathbb{R}^n$, there exists a unique solution $\varphi:[0,T] \to \mathbb{R}^n$ to the problem \eqref{caratheq}–\eqref{carathin}.
\end{theorem}

\begin{proof} By Proposition \ref{teonemytskii2}, for $\varphi \in C([0,T];\mathbb{R}^n)$, we have $\mathcal{N}_f(\varphi) \in L^p(0,T;\mathbb{R}^n)$. Since $p > \max\{1/\alpha_j : j \in \{1, \dots, n\}\}$, it follows from Theorem \ref{202504011518} that
$$
J^{\alpha_j}_t \big( \mathcal{N}_f(\varphi) \big)_j = J^{\alpha_j}_t f_j(\varphi(\cdot), \cdot) \in C([0,T];\mathbb{R}),
$$
for every $1 \leq j \leq n$. The same conclusion holds in the classical case $\alpha_j=1, \ j \in \{1, \dots , n\}$, by standard arguments.

This allows us to define the operator $\mathcal{T}: C([0,T];\mathbb{R}^n) \to C([0,T];\mathbb{R}^n)$ by
\begin{equation}\label{202505061854}
\mathcal{T}\big(\varphi(t)\big) = \xi +
\left(
\begin{array}{c}
J^{\alpha_1}_t f_1(\varphi(t),t) \\
\vdots \\
J^{\alpha_n}_t f_n(\varphi(t),t)
\end{array}
\right),
\end{equation}
or more explicitly, for each $1 \leq j \leq n$,
\begin{equation*}
\Big(\mathcal{T} 
\big(
\varphi(t)
\big)\Big)_j
=
\xi_j
+
\frac{1}{\Gamma(\alpha_j)} \int_0^t (t-s)^{\alpha_j-1} f_j(\varphi(s),s)\,ds.
\end{equation*}

Set $\alpha_0:=\min\{\alpha_j:j\in\{1,\ldots,n\}\}$ and consider
$$M:=\sum_{j=1}^n \dfrac{T^{\alpha_j-\alpha_0}\Gamma(\alpha_0)}{\Gamma(\alpha_j)}.$$
Now, observe that if $\varphi, \psi \in C([0,T];\mathbb{R}^n)$, $(C_2)$ ensures that
\begin{equation} \label{202504282018}
\left|\Big(\mathcal{T} 
\big(
\varphi(t)
\big)\Big)_j-\Big(\mathcal{T} 
\big(
\psi(t)
\big)\Big)_j\right|
 \leq J^{\alpha_j}_t \Big[\ell(t) \left\| \varphi(t) - \psi(t) \right\|_{\mathbb{R}^n} \Big],
\end{equation}
for every $t\in[0,T]$.

Since $\alpha_0 \leq \alpha_j$ for all $1 \leq j \leq n$, it follows that, for every $0 \leq s \leq t \leq T$,
\[
(t-s)^{\alpha_j-1} = (t-s)^{\alpha_j-\alpha_0}(t-s)^{\alpha_0-1}.
\]
Hence, from \eqref{202504282018} we have
\begin{equation*}
\left|\Big(\mathcal{T} 
\big(
\varphi(t)
\big)\Big)_j-\Big(\mathcal{T} 
\big(
\psi(t)
\big)\Big)_j\right| 
\leq 
\left[\dfrac{T^{\alpha_j-\alpha_0}\Gamma(\alpha_0)}{\Gamma(\alpha_j)}\right] J_t^{\alpha_0} \Big[\ell(t) \left\| \varphi(t) - \psi(t) \right\|_{\mathbb{R}^n} \Big],
\end{equation*}
for every $t \in [0,T]$, and consequently,
$$\left\|\mathcal{T}\big(\varphi(t)\big)-\mathcal{T}\big(\psi(t)\big)\right\|_{\mathbb{R}^n}\leq J_t^{\alpha_0} \Big[M\ell(t) \left\| \varphi(t) - \psi(t) \right\|_{\mathbb{R}^n} \Big],$$
for every $t\in[0,T]$. 

This allows us to obtain that
\begin{multline*}
\left\|\mathcal{T}^2\big(\varphi(t)\big)-\mathcal{T}^2\big(\psi(t)\big)\right\|_{\mathbb{R}^n} 
\leq J_t^{\alpha_0} \Big[M\ell(t) \left\| \mathcal{T}(\varphi(t)) - \mathcal{T}(\psi(t)) \right\|_{\mathbb{R}^n} \Big]
\\\leq J_t^{\alpha_0} \Big\{M\ell(t) J_t^{\alpha_0} \Big[M\ell(t) \left\| \varphi(t) - \psi(t) \right\|_{\mathbb{R}^n} \Big] \Big\}
\end{multline*}
for every $t\in[0,T]$. From this iterated procedure, we deduce for any $k \in \mathbb{N}$ that
\begin{multline*}
\left\|\mathcal{T}^k\big(\varphi(t)\big)-\mathcal{T}^k\big(\psi(t)\big)\right\|_{\mathbb{R}^n} 
\\ \leq \left\| \varphi - \psi \right\|_{C([0,T];\mathbb{R}^n)} \underbrace{J^{\alpha_0}_t \Big\{M\ell(t) J^{\alpha_0}_t \Big[ M\ell(t) J^{\alpha_0}_t\Big(\cdots M\ell(t)J^{\alpha_0}_t M\ell(t)\Big) \Big] \Big\}}_{k \text{-times}},
\end{multline*}
for every $t\in[0,T]$. Then, thanks to Lemma \ref{lemaint}, we deduce that there exists $k_0 \in \mathbb{N}$ such that $\mathcal{T}^{k_0}$ is a contraction. Therefore, by applying the Banach Fixed Point Theorem, we obtain the existence and uniqueness of $\varphi \in C([0,T];\mathbb{R}^n)$, which is the fixed point of $\mathcal{T}$. It follows directly from Proposition \ref{integralequ} that $\varphi$ is the unique solution to \eqref{caratheq}-\eqref{carathin}.
\end{proof}

We observe that when $\alpha_j = 1$ for every $j \in \{1, \dots, n\}$, we recover the classical case, which, unlike the setting in the previous theorem, requires only that $f$ be an $L^1$-Carathéodory function. This condition, however, is no longer sufficient when at least one $\alpha_j \in (0,1)$, as noted in Remark \ref{202504290836}.

To conclude this section, we summarize the importance of the assumptions made on the function $f$ in Theorem \ref{202504011725}:
\begin{itemize}
\item[(i)] The condition $p > \max\{1/\alpha_j : j \in \{1, \dots, n\}\}$ is essential to ensure that the operator $\mathcal{T}$, defined in \eqref{202505061854}, maps $C([0,T];\mathbb{R}^n)$ into itself.
\item[(ii)] If, in addition, conditions $(C_p^*)$ and $(C_2)$ are satisfied, their structures are enough for us to apply Banach's fixed point theorem, which ensures the uniqueness of the solution. 
\end{itemize}

\begin{remark} It is worth noting that, when $\alpha_j = 1$ for every $j \in \{1, \dots, n\}$, condition $(C_1)$ (which is naturally replaced by $(C_p^*)$ in the fractional framework; see Remark \ref{202505061925}) is sufficient to guarantee existence of a solution, while the additional Lipschitz-type condition $(C_2)$ becomes necessary to ensure uni\-que\-ness (see Theorems 1.44 and 1.45 in \cite{Ro1} for more details).  Although in this paper we have chosen to prove existence and uni\-que\-ness simultaneously under the stronger set of assumptions, we conjecture that assuming only $(C_p^*)$ would still suffice to establish existence of a solution, as in the classical setting. 
\end{remark}

Our goal now is to demonstrate that if condition $(C_p^*)$ is assumed in isolation and one of the differentiation orders equals $1/p$, the problem may fail to admit a solution. This highlights the critical relationship between the integrability condition in the $L^p$-Carathéodory framework and the strict inequality required for the fractional orders.

\begin{theorem}\label{202505021500} Let $p > 1$ and $f: \mathbb{R}^n \times [0,T] \to \mathbb{R}^n$ be a Carathéodory function that satisfies condition $(C_p^*)$. Suppose that $\{\alpha_j\}_{j=1}^n \subset (0,1]$ and that there exists $j_0 \in \{1, \ldots, n\}$ such that $\alpha_{j_0} = 1/p$. Then, the problem \eqref{caratheq}–\eqref{carathin} may not admit a solution $\varphi: [0,T] \to \mathbb{R}^n$.
\end{theorem}

\begin{proof} Let us prove the case $n=1$ for simplicity. First, we recall that Hardy and Littlewood showed in \cite[item (iv) of Section 3.5]{HaLi} that there exists $\sigma \in L^p(0,T;\mathbb{R})$ such that $J_t^{1/p} \sigma(t)$ is unbounded. 

Now, consider the function $f: \mathbb{R} \times [0,T] \to \mathbb{R}$ defined by
$$
f(x,t) = x + \sigma(t).
$$

It is straightforward to verify that $f$ is a Carathéodory function that satisfies condition $(C_p^*)$. 

Next, assume that there exists a solution $\phi(t)$ to the problem \eqref{caratheq}–\eqref{carathin}. By applying $J_t^{1/p}$ to both sides of \eqref{caratheq} and using Proposition \ref{202504011515}, we obtain the following equation:
$$
\phi(t) - \phi(0) - J_t^\alpha \phi(t) = J_t^\alpha \sigma(t),
$$
for all $t \in [0,T]$. Since Theorem \ref{202504011518} guarantees that $J_t^\alpha \phi(t)$ is continuous, it follows that $J_t^\alpha \sigma(t)$ must also be continuous on $[0,T]$ and, consequently, bounded. This leads to a contradiction, as we know that $J_t^{1/p} \sigma(t)$ is unbounded. Therefore, no solution $\phi(t)$ exists for this problem.
\end{proof}

\end{document}